\documentclass[12pt]{amsart}
\usepackage{amsmath,amssymb,amsfonts,amsthm}

\usepackage[alphabetic]{amsrefs}
\usepackage{bm}
\usepackage{graphicx}
\usepackage{ascmac}
\usepackage[all]{xy}

\usepackage[colorlinks, linkcolor=blue,anchorcolor=blue,citecolor=green]{hyperref}
\usepackage{graphics,epic}

\usepackage{times}
\usepackage{tikz}
\usetikzlibrary{positioning,arrows.meta}
\usetikzlibrary{matrix,calc}
\usetikzlibrary{decorations.pathmorphing}

\newtheorem{theorem}{Theorem}[section]
\newtheorem*{theorem*}{Theorem}

\newtheorem{lemma}[theorem]{Lemma}
\newtheorem{proposition}[theorem]{Proposition}
\newtheorem{corollary}[theorem]{Corollary}

\newtheorem*{conjecture*}{Conjecture}

\newtheorem{remark}[theorem]{Remark}
\newtheorem{definition}[theorem]{Definition}

\newcommand{\opname}[1]{\operatorname{\mathsf{#1}}}

\renewcommand{\mod}{\opname{mod}\nolimits}

\newcommand{\im}{\opname{im}\nolimits}
\renewcommand{\ker}{\opname{ker}\nolimits}

\newcommand{\id}{\mathbf{1}}

\newcommand{\Ext}{\opname{Ext}}

\begin{document}
	
\title{On higher extensions of quiver representations over $\mathbb{F}_1$}\thanks{This work was supported by the National Natural Science Foundation of China (Grant No. 12571040)}
\author{Changjian Fu}
\author{Liang Yang}
\author{Zhiyuan Zeng}
\address{Department of Mathematics, Sichuan University, Chengdu, 610064 PR China}

\email{changjianfu@scu.edu.cn(Fu)}
\email{malyang@scu.edu.cn(Yang)}	
\email{2697413645@qq.com(Zeng)}

\begin{abstract}
We show that higher extension spaces between finite-dimensional nilpotent $\mathbb{F}_1$-representations maybe  infinite-dimensional, thereby clarifying a misconception in the  literature. Our examples arise from cyclic quivers. In particular, for a cyclic quiver $\Delta_n$, we show that $\Ext^3(-,-)$ vanishes for any pair of finite-dimensional nilpotent $\mathbb{F}_1$-representations of $\Delta_n$, while $\Ext^2(-,-)$ is infinite-dimensional for any pair of simple representations.

\end{abstract}
\maketitle

\section{Introduction}
 Let $\mathbb{F}_1$ be the virtual field, i.e., the ``field" with one element. Szczesny \cite{Sz12} introduced quiver representations 
over $\mathbb{F}_1$ as a degenerated combinatorial model of quiver  representations over a field. Over the past decade, various concepts and techniques from the classical theory  of quiver representations over a field have been developed in the $\mathbb{F}_1$-setting, and the connection between the combinatorics of
 $\mathbb{F}_1$-representations and the representation theory over a field has been extensively explored; see, for example \cites{JS23,JS24,Kleinau25}.

 Let $Q$ be a finite quiver.
 The category $\opname{rep}(Q,\mathbb{F}_1)$ of finite-dimensional $\mathbb{F}_1$-representations of $Q$ behaves in many ways like a non-additive
 version of a module category. In particular, one may adapt Yoneda's construction to define the extension space $\Ext^n(\mathbf{L},\mathbf{N})$ for  any $\mathbf{L},\mathbf{M}\in \opname{rep}(Q,\mathbb{F}_1)$ and any positive integer $n$. Note that $\Ext^n(\mathbf{L},\mathbf{N})$ is not a group but rather a pointed set, equivalently, an $\mathbb{F}_1$-vector space.
 
It follows directly from the definition that $\Ext^1(\mathbf{L},\mathbf{N})$ is a finite pointed set whenever $\mathbf{L}$ and $\mathbf{N}$ are finite-dimensional representations, which is essential for the definition of Hall algebra of $\opname{rep}(Q,\mathbb{F}_1)$. However, the situation for higher extension spaces is far from clear, and there exists a common misconception in the literature regarding their finiteness; see, for instance, \cite{Sz12}*{Section 4}. The aim of this short note is to show that the finiteness does not hold in general. In particular, we exhibit that the second  extension space $\Ext^2(\mathbf{L},\mathbf{N})$ is infinite-dimensional for any pair of simple representations over a cyclic quiver, thereby clarifying the misconception.

\section{Preliminary}
\subsection{Quiver representations over \texorpdfstring{$\mathbb{F}_1$}{F1}}
In this section, we recall the basic definitions and properties of quiver representations over the virtual field $\mathbb{F}_1$.  We follow \cites{Sz12, JS23} and refer to \cite{Sz12} for unexplained definitions related to representations of quivers over $\mathbb{F}_1$.

A finite-dimensional {\it $\mathbb{F}_1$-vector space} is a finite pointed set $V:=(V,0_V)$. We say that $V$ is of dimension $\dim V:=|V|-1$. An {\it $\mathbb{F}_1$-linear map} from $V:=(V,0_V)$ to $W:=(W,0_W)$  is a pointed function $f:V\to W$  such that $f|_{V\backslash f^{-1}(0_W)}$ is an injection.  We denote by $\mathbf{k}$ the one-dimensional $\mathbb{F}_1$-vector space.

A {\it quiver} is a quadruple $Q=(Q_0,Q_1,s,t)$ consisting of a set of vertices $Q_0$, a set of arrows $Q_1$, and two maps $s,t:Q_1\to Q_0$ that assign to each arrow $\alpha\in Q_1$ its source $s(\alpha)$ and target $t(\alpha)$.
The quiver $Q$ is {\it finite} if both $Q_0$ and $Q_1$ are finite sets.

Let $Q$ be a finite quiver. A {\it representation} of $Q$ over $\mathbb{F}_1$ is a collection \[\mathbf{M}=(M_i, M_\alpha)_{i\in Q_0,\alpha\in Q_1},\] where
\begin{itemize}
    \item $M_i$ is a finite-dimensional $\mathbb{F}_1$-vector space for each vertex $i\in Q_0$;
    \item $M_\alpha:M_{s(\alpha)}\to M_{t(\alpha)}$ is an $\mathbb{F}_1$-linear map for each arrow $\alpha\in Q_1$.
\end{itemize}
A representation $\mathbf{M}=(M_i,M_\alpha)$ is {\it nilpotent}, if there exists $N\geq 0$ such that for any $n\geq N$, $M_{\alpha_n}\circ\cdots\circ M_{\alpha_1}=0$ for any path $\alpha_1\alpha_2\cdots \alpha_n$ in $Q$.

Let $\mathbf{M}=(M_i,M_\alpha)$ and $\mathbf{N}=(N_i,N_\alpha)$ be representations of $Q$ over $\mathbb{F}_1$. A morphism $\Phi:\mathbf{M}\to \mathbf{N}$ is a collection of $\mathbb{F}_1$-linear maps $\Phi=(\phi_i)_{i\in Q_0}$, where $\phi_i:M_i\to N_i$, such that the following diagram commutes for each arrow $\alpha:i\to j\in Q_1$:
\[
\xymatrix{M_i\ar[r]^{M_\alpha}\ar[d]^{\phi_i}&M_j\ar[d]^{\phi_j}\\
 N_i\ar[r]^{N_\alpha}&N_j.}
\]
The kernel $\ker \Phi$ and image $\opname{im} \Phi$ of $\Phi$ can be defined as usual, which are subrepresentations of $\mathbf{M}$ and $\mathbf{N}$ respectively (cf. \cite{Sz12}*{Section 4}).

Let $\opname{rep}(Q,\mathbb{F}_1)$ be the category of representations of $Q$ over $\mathbb{F}_1$ and $\opname{rep}(Q,\mathbb{F}_1)_{\rm nil}$ the full subcategory of $\opname{rep}(Q,\mathbb{F}_1)$ consisting of nilpotent representations. Although the category $\opname{rep}(Q,\mathbb{F}_1)$ is not additive, it shares many good properties with the category of representations of $Q$ over a field. In particular, both the
Jordan–H\"{o}lder and Krull–Schmidt theorems hold in $\opname{rep}(Q,\mathbb{F}_1)$; see \cite{Sz12}.

\subsection{Yoneda's Extension} 
In this section, we recall Yoneda's extension for $\opname{rep}(Q,\mathbb{F}_1)$, following \cite{FRY24}.

Denote by $\mathbf{0}$ the zero object of $\opname{rep}(Q,\mathbb{F}_1)$.
An exact sequence of length $n+2$ of $\opname{rep}(Q,\mathbb{F}_1)$ is a sequence of morphism
\[
\xymatrix{\mathbf{L}\ \ar[r]^{\alpha_0}&\mathbf{M}_1\ar[r]^{\alpha_1}&\mathbf{M}_2\ar[r]^{\alpha_2}&\cdots\ar[r]^{\alpha_{n-1}}&\mathbf{M}_n\ar[r]^{\alpha_n}&\mathbf{N}}
\]
 such that 
 \begin{itemize}
     \item $\ker \alpha_{i+1}=\opname{im} \alpha_i$ for $0\leq i\leq n-1$;
     \item $\alpha_0$ is a monomorphism, equivalently, $\ker\alpha_0=\mathbf{0}$;
     \item $\alpha_n$ is an epimorphism, equivalently, $\opname{im} \alpha_n=\mathbf{N}$.
 \end{itemize}  We will use $\rightarrowtail$ (resp. $\twoheadrightarrow$) to indicate a morphism is a monomorphism  (resp. an epimorphism).

Let 
\[\xymatrix{\epsilon: &\mathbf{L}\ \ar@{>->}[r]^{\alpha_0}&\mathbf{M}_1\ar[r]^{\alpha_1}&\mathbf{M}_2\ar[r]^{\alpha_2}&\cdots\ar[r]^{\alpha_{n-1}}&\mathbf{M}_n\ar@{->>}[r]^{\alpha_n}&\mathbf{N}}
\]
and 
\[\xymatrix{\epsilon': &\mathbf{U}\ \ar@{>->}[r]^{\beta_0}&\mathbf{N}_1\ar[r]^{\beta_1}&\mathbf{N}_2\ar[r]^{\beta_2}&\cdots\ar[r]^{\beta_{n-1}}&\mathbf{N}_n\ar@{->>}[r]^{\beta_n}&\mathbf{V}}
\]
be  exact sequences of length $n+2$  in $\opname{rep}(Q,\mathbb{F}_1)$. 
Denote by $\epsilon\oplus \epsilon'$ the following exact sequence 
\[
\xymatrix{\mathbf{L}\oplus \mathbf{U}\ \ar@{>->}[r]^-{\tiny  \begin{pmatrix}\alpha_0&0\\0&\beta_0\end{pmatrix}}& \mathbf{M}_1\oplus \mathbf{N}_1\ar[r]&\cdots\ar[r]&\mathbf{M}_n\oplus \mathbf{N}_n\ar@{->>}[r]^-{\tiny\begin{pmatrix}\alpha_n&0\\0&\beta_n\end{pmatrix}}&\mathbf{N}\oplus \mathbf{V},}
\]
which is called the direct sum of $\epsilon$ and $\epsilon'$.

Let $\mathbb{E}^n(\mathbf{N},\mathbf{L})$ be the set of exact sequences of length $n+2$ that start at $\mathbf{L}$ and end at $\mathbf{N}$. An element in $\mathbb{E}^n(\mathbf{N},\mathbf{L})$ is called an {\it $n$-extension} of $\mathbf{N}$ by $\mathbf{L}$. We say that an element in $\mathbb{E}^n(\mathbf{N},\mathbf{L})$ is non-zero if at least one of its $n+2$ components is non-zero.
\begin{definition}
A nonzero exact sequence $\epsilon\in \mathbb{E}^n(N,L)$ is called primitive if $\epsilon$ has no non-trivial direct sum decomposition, that is, if $\epsilon= \epsilon_1\oplus \epsilon_2$ for exact sequences $\epsilon_1$ and $\epsilon_2$ of length $n+2$, then $\epsilon_1= 0$ or $\epsilon_2= 0$.
\end{definition}
\begin{remark}
Let $\mathbf{L},\mathbf{N}\in \opname{rep}(Q,\mathbb{F}_1)$. Every nonzero exact sequence in $\mathbb{E}^n(\mathbf{N},\mathbf{L})$ is a finite direct sum of primitive exact sequences.
\end{remark}
Recall that the $n$-extensions $\epsilon$ and $\epsilon'$ of $\mathbf{N}$ by $\mathbf{L}$ satisfy the relation $\epsilon\leadsto\epsilon'$ if there is a commutative diagram
\[
\xymatrix{\epsilon: &\mathbf{L}\ \ar@{=}[d]\ar@{>->}[r]&\mathbf{E}_n\ar[d]^{f_n}\ar[r]&\cdots\ar[r]&\mathbf{E}_1\ar[d]^{f_1}\ar@{->>}[r] &\mathbf{N}\ar@{=}[d]\\
\epsilon':&\mathbf{L}\ \ar@{>->}[r]&\mathbf{E}_n'\ar[r]&\cdots\ar[r]&\mathbf{E}_1'\ar@{->>}[r] &\mathbf{N}.
}
\]
 Moreover, if $f_1,\dots, f_n$ are isomorphisms, we say that $\epsilon$ is isomorphic to $\epsilon'$, and denote it by $\epsilon\cong \epsilon'$.
The relation $\leadsto$  generates an equivalence relation on $\mathbb{E}^n(\mathbf{N},\mathbf{L})$. We denote by $[\epsilon]$ the equivalence class of the $n$-extension $\epsilon$, and by $\Ext^n(\mathbf{N},\mathbf{L})$ the set of all equivalence classes of $n$-extensions of $\mathbf{N}$ by $\mathbf{L}$. The set $\Ext^n(\mathbf{N},\mathbf{L})$ is a pointed set with
\[0=[\xymatrix{\mathbf{L}\ar@{=}[r]&\mathbf{L}\ar[r]&\mathbf{0}\ar[r]&\cdots\ar[r]&\mathbf{0}\ar[r]&\mathbf{N}\ar@{=}[r]&\mathbf{N}}]
\] for $n\geq 2$ and 
\[\xymatrix{0=[\mathbf{L}\ \ar@{>->}[r]^-{\tiny\begin{pmatrix}1_{\mathbf{L}}&0\end{pmatrix}}&\mathbf{L}\oplus \mathbf{N}\ar@{->>}[r]^-{\tiny\begin{pmatrix}0\\ 1_{\mathbf{N}}\end{pmatrix}}&\mathbf{N}]}
\]
for $n=1$. 
The following is obvious, see \cite{FRY24}*{Lemma 2.4}.
\begin{lemma}\label{l:equiv=0}
\begin{itemize}
    \item[(1)] For any exact sequences $\epsilon$ and $\delta$ of length $n+2$. If $[\epsilon]=0$ and $[\delta]=0$, then $[\epsilon\oplus \delta]=0$.
    \item[(2)] For any $n\geq 1$ and $\mathbf{L}\in \opname{rep}(Q,\mathbb{F}_1)$, $\Ext^n(\mathbf{0},\mathbf{L})=0=\Ext^n(\mathbf{L},\mathbf{0})$. If $\Ext^n(\mathbf{L},-)=0$, then $\Ext^{n+1}(\mathbf{L},-)=0$.
\end{itemize}

\end{lemma}

% \begin{definition}
%    The category $\opname{rep}(Q,\mathbb{F}_1)$ is of global dimension $n$ if $\Ext^n(-,-)\neq 0$ and $\Ext^{n+1}(-,-)=0$. In this case, we denote it by $gl.dim \opname{rep}(Q,\mathbb{F}_1)=n$. It is hereditary if $gl.dim \opname{rep}(Q,\mathbb{F}_1)\leq 1$. 
% \end{definition}

% \begin{definition}
%     A representation $P\in \opname{rep}(Q,\mathbb{F}_1)$ is projective if for any epimorphism $f:M\twoheadrightarrow N$  and any morphism $g:P\to N$ of $\opname{rep}(Q,\mathbb{F}_1)$, there is a morphism $h:P\to M$ such that $g=f\circ h$.
% \end{definition}

% The following is clear.
% \begin{lemma}
% Let $P\in \opname{rep}(Q,\mathbb{F}_1)$ be a projective representation.  Then $\Ext^1(P,-)=0$ and each direct summand of $P$ is projective.

% \end{lemma}
% \begin{remark}
% In general, the condition $\Ext^1(P,-)=0$ does not imply that $P$ is projective and the direct sum of two projective representations is not projective, see example in Section \ref{s:non-linear}. 
% \end{remark}

\section{Higher extensions over cyclic quivers}
\subsection{The category \texorpdfstring{$\opname{rep}(\Delta_n,\mathbb{F}_1)_{\rm nil}$}{rep(Delta,F1)}}
Throughout this section, let $\Delta_n$ be the cyclic quiver with $n$ vertices, where $n\geq 1$. We label the vertex set by $\{1,2,\ldots, n\}$ such that the arrows are precisely from vertex $i$ to $i+1$ (taken modulo $n$)\footnote{It is convenient to denote a vertex  by an integer larger than $n$. In this case, it should be read as modulo $n$.}. 
Recall that $\opname{rep}(\Delta_n,\mathbb{F}_1)_{\rm nil}$ is the category of finite-dimensional nilpotent representations of $\Delta_n$ over $\mathbb{F}_1$, which is essentially identical to that over a field.

For each  vertex $k$ of $\Delta_n$ and a positive integer $r$, let $V=\{0,v_k,v_{k+1},\cdots, v_{k+r-1}\}$ be an $r$-dimensional $\mathbb{F}_1$-vector space.
Let $\mathbf{I}_{[k,r]}:=(V_i,V_\alpha)$ be the $\mathbb{F}_1$-representation of $\Delta_n$, where
\begin{itemize}
    \item $V_j=\{0,v_i\mid  i-j\equiv 0 (\mod n), k\leq i\leq k+r-1\}$;
    \item if $\alpha:j\to j+1$, then $V_\alpha(v_i)=v_{i+1}$ for $v_i\in V_j$. Here we use the convention $v_{k+r}=0$.
\end{itemize}
The following proposition summarizes certain basic structural results on $\opname{rep}(\Delta_n,\mathbb{F}_1)$, which can be deduced from the definitions and the construction of $\mathbf{I}_{[k,r]}$; see \cite{Sz12}*{Section 11}.
\begin{proposition}\label{prop:structure-nil-ind}
    \begin{itemize}
        \item[(1)] The set $\{\mathbf{I}_{[k,r]}\mid k\in (\Delta_n)_0, 1\leq r\in \mathbb{N}\}$ forms a complete set of non-isomorphic finite-dimensional indecomposable nilpotent representations of $\Delta_n$.
        \item[(2)] The set $\{\mathbf{I}_{[k,1]}\mid k\in (\Delta_n)_0\}$ are precisely the simple objects in $\opname{rep}(\Delta_n,\mathbb{F}_1)_{\rm nil}$.
        \item[(3)] Each indecomposable representation $\mathbf{I}_{[k,r]}$ has a unique composition series whose composition factors are $\mathbf{I}_{[k+r-1,1]}, \mathbf{I}_{[k+r-2,1]},\ldots, \mathbf{I}_{[k,1]}$. In particular,  its unique top is $\mathbf{I}_{[k,1]}$,  while its unique socle is $I_{[k+r-1,1]}$.
        \item[(4)] Each nonzero subrepresention or quotient representation of $\mathbf{I}_{[k,r]}$ is also indecomposable.
        \item[(5)] There is an epimorphism from $\mathbf{I}_{[i,l]}$ to $\mathbf{I}_{[k,r]}$ if and only if $i=k$ and $l\geq r$. Dually, there is a monomorphism from $\mathbf{I}_{[i,l]}$ to $\mathbf{I}_{[k,r]}$ if and only if $i+l-1\equiv k+r-1(\mod n)$ and $l\leq r$. Moreover, in both cases, the morphism is unique.
        \item[(6)] Let $\mathbf{L},\mathbf{M},\mathbf{N}$ be indecomposable representations in $\opname{rep}(\Delta_n,\mathbb{F}_1)$, $f:\mathbf{L}\to \mathbf{M}$ and $g:\mathbf{M}\to \mathbf{N}$ be morphisms. If $g\circ f$ is injective, then both $f$ and $g$ are injective. If $g\circ f$ is surjective, then  both $f$ and $g$ are surjective.
    \end{itemize}
\end{proposition}
We also have a splitting lemma as \cite{FRY24}*{Lemma 3.2}. Here we give a structural proof.
\begin{lemma}\label{l:decom-epi}
Let $f:\mathbf{M}\to \mathbf{N}$ be an epimorphism. Let $\mathbf{N}_1,\mathbf{N}_2$ be subrepresentations of $\mathbf{N}$ such that $\mathbf{N}=\mathbf{N}_1\oplus \mathbf{N}_2$ and  $\mathbf{N}_1\cong \mathbf{I}_{[k,r]}$ for some $k$ and $r$. Then there exist subrepresentations $\mathbf{M}_1,\mathbf{M}_2$ of $\mathbf{M}$ such that 
\begin{enumerate}
\item $\mathbf{M}=\mathbf{M}_1\oplus \mathbf{M}_2$;
\item $\mathbf{M}_1\cong \mathbf{I}_{[k,l]}$ for some $l\geq r$;
\item $f$ can be rewritten as $\mathbf{M}_1\oplus \mathbf{M}_2\xrightarrow{\tiny\begin{pmatrix}f_1&0\\ 0&f_2\end{pmatrix}}\mathbf{N}_1\oplus \mathbf{N}_2$, where $f_1$ is the unique epimorphism from $\mathbf{I}_{[k,l]}$ to $\mathbf{I}_{[k,r]}$ and $f_2$ is an epimorphism.
\end{enumerate}
\end{lemma}
\begin{proof}
    According to Proposition \ref{prop:structure-nil-ind} (3), if $g:\mathbf{I}_{[i,l]}\to \mathbf{I}_{[k,r]}$ is a nonzero morphism, then the unique socle of $\mathbf{I}_{[k,r]}$ belongs to the image $\im g$. It follows that a map $\mathbf{I}_{[i,l]}\oplus \mathbf{I}_{[j,m]}\xrightarrow{(g_1,g_2)}\mathbf{I}_{[k,r]}$ can not be a morphism whenever both $g_1$ and $g_2$ are nonzero.
    
    Applying the Krull-Schmidt theorm to $\mathbf{M}$, we may write $\mathbf{M}=\mathbf{L}_1\oplus \cdots\oplus \mathbf{L_t}$, where each $\mathbf{L}_i$ is indecomposable. Denote by $l_i:\mathbf{L}_i\to \mathbf{M}$ the inclusion, and $p_1:\mathbf{N}\to \mathbf{N}_1$ the projection. We conclude that there exists a unique $i$ such that $p_1\circ f\circ l_i\neq 0$. Set $\mathbf{M}_1=\mathbf{L}_i$ and $\mathbf{M}_2=\bigoplus_{j\neq i}\mathbf{L}_j$. Then $\mathbf{M}=\mathbf{M}_1\oplus \mathbf{M}_2$ and $f$ can be rewritten as $\mathbf{M}_1\oplus \mathbf{M}_2\xrightarrow{\tiny\begin{pmatrix}f_1&0\\ 0&f_2\end{pmatrix}}\mathbf{N}_1\oplus \mathbf{N}_2$. It turns out that $f_1$ and $f_2$ are epimorphisms. Now the result follows Proposition \ref{prop:structure-nil-ind} (5).
\end{proof}
\begin{corollary}\label{c:decom-seq}
Let $\mathbf{N}=\bigoplus_{i=1}^t\mathbf{N}_i$ with indecomposable direct summands $\mathbf{N}_i$, $1\leq i\leq t$. Every exact sequence $\epsilon$ ending at $\mathbf{N}$ of length $n+2$ has a decomposition
\[
\epsilon=\epsilon_1\oplus \cdots\oplus \epsilon_t,
\]
where $\epsilon_i$ is an exact sequence  of length $n+2$ ending at $\mathbf{N}_i$.
\end{corollary}

\subsection{Non-finiteness of \texorpdfstring{$\Ext^2(-,-)$}{Ext2}}
\begin{lemma}\label{lem:non-zero-extension-r}
    Let $k$ be a vertex of $\Delta_n$, and $s$ a positive integer. For a nonnegative integer $r$ such that $rn+s\geq 2$, let $\epsilon_r$ be the following exact sequence of length $4$:
   \[ \xymatrix{\epsilon_r:& \mathbf{I}_{[k+s,1]}\ \ar@{>->}[r]&\mathbf{I}_{[k+1,rn+s]}\ar[r]&\mathbf{I}_{[k,rn+s]}\ar@{->>}[r]&\mathbf{I}_{[k,1]}.}
    \]
    Assume that there is a commutative diagram of exact sequences
    \begin{equation}\label{diagram:comutaive-diag-morphism-in}
    \xymatrix{\delta: &\mathbf{I}_{[k+s,1]}\ \ar@{=}[d]\ar@{>->}[r]&\mathbf{M}\ar[d]\ar[r] &\mathbf{N}\ar[d]\ar@{->>}[r]&\mathbf{I}_{[k,1]}\ar@{=}[d]\\
    \epsilon_r:& \mathbf{I}_{[k+s,1]}\ \ar@{>->}[r]&\mathbf{I}_{[k+1,rn+s]}\ar[r]&\mathbf{I}_{[k,rn+s]}\ar@{->>}[r]&\mathbf{I}_{[k,1]}}
    \end{equation}
    or 
    \begin{equation}\label{diag:morphism-out}
\xymatrix{\delta:&\mathbf{I}_{[k+s,1]}\ \ar@{=}[d]\ar@{>->}[r]&\mathbf{M}\ar[r] &\mathbf{N}\ar@{->>}[r]&\mathbf{I}_{[k,1]}\ar@{=}[d]\\
    \epsilon_r:&\mathbf{I}_{[k+s,1]}\ \ar@{>->}[r]&\mathbf{I}_{[k+1,rn+s]}\ar[u]\ar[r]&\mathbf{I}_{[k,rn+s]}\ar[u]\ar@{->>}[r]&\mathbf{I}_{[k,1]}.}
    \end{equation}
    Then $\delta\cong\epsilon_r\oplus \gamma$, where $\xymatrix{\gamma:\mathbf{0}\,\ar@{>->}[r]&\mathbf{U}\ar[r]^{1_\mathbf{U}}&\mathbf{U}\ar@{->>}[r]&\mathbf{0}}$ for some $\mathbf{U}$.
\end{lemma}
\begin{proof}
 Let us prove the case \eqref{diagram:comutaive-diag-morphism-in}, the case \eqref{diag:morphism-out} can be proved in a similar way. Assume that we are in the situation \eqref{diagram:comutaive-diag-morphism-in}.
  By Lemma \ref{l:decom-epi}, $\delta$ is either isomorphic to 
  \begin{equation}\label{seq:true}
      \xymatrix{\mathbf{I}_{[k+s,1]}\ \ar@{>->}[r]^{\scriptsize\begin{pmatrix}i\\0\end{pmatrix}}&\mathbf{\bar{M}}\oplus \mathbf{\bar{N}}\ar[r]^{\scriptsize\begin{pmatrix}f&0\\ 0&\id_{\mathbf{\bar{N}}}\end{pmatrix}}&\mathbf{I}_{[k,l]}\oplus \mathbf{\bar{N}}\ar@{->>}[r]^{\hspace{1em}\scriptsize\begin{pmatrix}\pi&0\end{pmatrix}}&\mathbf{I}_{[k,1]}} 
  \end{equation}
  or
  \begin{equation}\label{seq:false}
  \xymatrix{\mathbf{I}_{[k+s,1]}\ \ar@{>->}[r]^{\hspace{-1.5em}\scriptsize\begin{pmatrix}i\\0\end{pmatrix}}&\mathbf{\bar{M}}\oplus \mathbf{I}_{[k+1,l-1]}\ar[r]^{\scriptsize\begin{pmatrix}g&0\\ 0&h\end{pmatrix}}&\mathbf{\bar{N}}\oplus \mathbf{I}_{[k,l]} \ar@{->>}[r]^{\scriptsize\begin{pmatrix}0&\pi\end{pmatrix}}&\mathbf{I}_{[k,1]},}
  \end{equation}
  where $l\geq 1$, $\pi$ is the unique epimorphism from $\mathbf{I}_{[k,l]}$ to $\mathbf{I}_{[k,1]}$, $h$ is the unique monomorphism from $\mathbf{I}_{[k+1,l-1]}$ to $\mathbf{I}_{[k,l]}$, $i$ is a monomorphism, $g$ is an epimorphism and $\im f\cong \mathbf{I}_{[k+1,l-1]}$. 
  
  If $\delta$ is isomorphic to the exact sequence \eqref{seq:false}, then the commutative diagram \eqref{diagram:comutaive-diag-morphism-in} can be rewritten as
  \[
\xymatrix{\mathbf{I}_{[k+s,1]}\ar@{>->}[r]^{\hspace{-1.5em}\scriptsize\begin{pmatrix}i\\0\end{pmatrix}}\ar@{=}[d]&\mathbf{\bar{M}}\oplus \mathbf{I}_{[k+1,l-1]}\ar[d]^{\scriptsize\begin{pmatrix}\alpha_1&\alpha_2\end{pmatrix}}\ar[r]^{\scriptsize\begin{pmatrix}g&0\\ 0&h\end{pmatrix}}&\mathbf{\bar{N}}\oplus \mathbf{I}_{[k,l]}\ar[d]^{\scriptsize\begin{pmatrix}\beta_1&\beta_2\end{pmatrix}} \ar@{->>}[r]^{\scriptsize\begin{pmatrix}0&\pi\end{pmatrix}}&\mathbf{I}_{[k,1]}\ar@{=}[d]\\
  \mathbf{I}_{[k+s,1]}\ar@{>->}[r]&\mathbf{I}_{[k+1,rn+s]}\ar[r]&\mathbf{I}_{[k,rn+s]}\ar@{->>}[r]&\mathbf{I}_{[k,1]}.}
  \]
  According to Proposition \ref{prop:structure-nil-ind} (6), $\beta_2$ is an epimorphism, and $\beta_1=0$ by the definition of a morphism. Consequently, $l\geq rn+s\geq 2$. Note that $\beta_2\circ h\neq 0$ since  $\dim \im \beta_2\circ h=rn+s-1\geq 1$. It follows that $\alpha_2\neq 0$, and hence $\alpha_1=0$. This leads to a contradiction with the commutativity of the first square. 
  
  Thus, $\delta$ turns out to be isomorphic to the exact sequence \eqref{seq:true}. Again by the commutativity of the right square of \eqref{diagram:comutaive-diag-morphism-in} and Proposition \ref{prop:structure-nil-ind} (6), we know that $l\geq rn+s\geq 2$. In particular, $\mathbf{I}_{[k+1,l-1]}\neq \mathbf{0}$.

  We claim that $\mathbf{\bar{M}}$ is indecomposable. Otherwise, the short exact sequence \[\xymatrix{\mathbf{I}_{[k+s,1]}\ \ar@{>->}[r]^i&\mathbf{\bar{M}}\ar@{->>}[r]^{\hspace{-2.5em}\bar{f}}&\mathbf{I}_{[k+1,l-1]}\cong\im f}\] is split by Lemma \ref{l:decom-epi} again. It follows that the commutative diagram \eqref{diagram:comutaive-diag-morphism-in} can be rewritten as
  \[
  \xymatrix{
  \mathbf{I}_{[k+s,1]}\ar@{=}[d]\ar@{>->}[r]^{\hspace{-3.5em}\scriptsize\begin{pmatrix}i\\0\\0\end{pmatrix}}&\mathbf{I}_{[k+s,1]}\oplus \mathbf{I}_{[k+1,l-1]}\oplus \mathbf{\bar{N}}\ar[d]^{\scriptsize\begin{pmatrix} \mu_1&\mu_2&\mu_3\end{pmatrix}}\ar[rr]^{\hspace{2.5em}\scriptsize\begin{pmatrix}0&h&0\\0&0&\id_{\mathbf{\bar{N}}}\end{pmatrix}}&&\mathbf{I}_{[k,l]}\oplus \mathbf{\bar{N}}\ar[d]^{\scriptsize\begin{pmatrix}\nu_1&\nu_2\end{pmatrix}}\ar@{->>}[r]^{\hspace{1em}\scriptsize\begin{pmatrix}\pi&0\end{pmatrix}}&\mathbf{I}_{[k,1]}\ar@{=}[d]\\
    \mathbf{I}_{[k+s,1]}\ar@{>->}[r]&\mathbf{I}_{[k+1,rn+s]}\ar[rr]&&\mathbf{I}_{[k,rn+s]}\ar@{->>}[r]&\mathbf{I}_{[k,1]}.
  }
  \] 
  According to Proposition \ref{prop:structure-nil-ind} (6), $\nu_1$ is an epimorphism, and hence $\nu_2=0$. Again, the image $\im \nu_1\circ h\neq 0$.
  By the commutativity of the middle square, we conclude that $\mu_2\neq 0$ and $\mu_1=\mu_3=0$. This leads to a contradiction with the commutativity of the first square. 

  We have proved that $\mathbf{\bar{M}}$ is indecomposable, it turns out that $\mathbf{\bar{M}}\cong \mathbf{I}_{[k+1,l]}$ and $l\equiv s (\mod n)$. Consequently, $l=tn+s$ for some $t\geq r$. It remains to show that $t=r$.
  Now,  the commutative diagram \eqref{diagram:comutaive-diag-morphism-in} can be rewritten as
  \[
  \xymatrix{
  \mathbf{I}_{[k+s,1]}\ar@{=}[d]\ar@{>->}[r]^{\hspace{-2.5em}\scriptsize\begin{pmatrix}i\\0\end{pmatrix}}&\mathbf{I}_{[k+1,tn+s]}\oplus \mathbf{\bar{N}}\ar[d]^{\scriptsize\begin{pmatrix} \eta_1&\eta_2\end{pmatrix}}\ar[rr]^{\scriptsize\begin{pmatrix}f&0\\0&\id_{\mathbf{\bar{N}}}\end{pmatrix}}&&\mathbf{I}_{[k,tn+s]}\oplus \mathbf{\bar{N}}\ar[d]^{\scriptsize\begin{pmatrix}\nu_1&\nu_2\end{pmatrix}}\ar@{->>}[r]^{\hspace{1em}\scriptsize\begin{pmatrix}\pi&0\end{pmatrix}}&\mathbf{I}_{[k,1]}\ar@{=}[d]\\
    \mathbf{I}_{[k+s,1]}\ar@{>->}[r]&\mathbf{I}_{[k+1,rn+s]}\ar[rr]&&\mathbf{I}_{[k,rn+s]}\ar@{->>}[r]&\mathbf{I}_{[k,1]}.
  }
  \]
  A similar discussion shows that $\eta_1\neq 0$ and $\eta_2=0$. If $t>r$, then $\eta_1\circ i=0$, which contradicts to the commutativity of the left square. Therefore, $t=r$, and we are done.
\end{proof}
\begin{theorem}
    For any pair of vertices $j$ and $k$, $\Ext^2(\mathbf{I}_{[j,1]},\mathbf{I}_{[k,1]})\cong \mathbb{N}$, where $\mathbb{N}$ is the set of nonnegative integers viewed as a pointed set with the base point $0$.
\end{theorem}
\begin{proof} 
We may assume that $j=k+s$ for some $0<s\leq n$.   We show that the nonzero elements in $\Ext^2(\mathbf{I}_{[k+s,1]},\mathbf{I}_{[k,1]})$  are in bijection with $\mathbb{N}$.

Let us first assume that $s\geq 2$.
By Lemma \ref{lem:non-zero-extension-r}, for any nonnegative integer $r$, $[\epsilon_r]$ is nonzero in $\Ext^2(\mathbf{I}_{[k+s,1]},\mathbf{I}_{[k,1]})$, and $[\epsilon_r]\neq [\epsilon_t]$ whenever $r\neq t$. It remains to show that for any nonzero $[\delta]\in \Ext^2(\mathbf{I}_{[k+s,1]},\mathbf{I}_{[k,1]})$, one has $[\delta]=[\epsilon_r]$ for some $r\geq 0$. Otherwise, by applying Lemma \ref{l:decom-epi}, $\delta$ is a direct sum of the following exact sequences of length $4$:
    \[
    \xymatrix{\mathbf{I}_{[k+s,1]}\,\ar@{>->}[r]&\mathbf{V}\ar[r]&\mathbf{W}\ar@{->>}[r]&\mathbf{0},}  \xymatrix{\mathbf{0}\, \ar@{>->}[r]&\mathbf{X}\ar[r]&\mathbf{Y}\ar@{->>}[r]&\mathbf{I}_{[k,1]}}
    \]
    and
    \[
    \xymatrix{\mathbf{0}\,\ar@{>->}[r]&\mathbf{Z}\ar[r]^{1_\mathbf{Z}}&\mathbf{Z}\ar@{->>}[r]&\mathbf{0}.}
    \]
    It turns out that $[\delta]=0$. Hence, $\Ext^2(\mathbf{I}_{[k+s,1]},\mathbf{I}_{[k,1]})=\{0,[\epsilon_r]\mid r\geq 0\}\cong \mathbb{N}$.

    Now assume that $s=1$. Again by Lemma \ref{lem:non-zero-extension-r}, for any positive integer $r$, $[\epsilon_r]\neq 0$ in $\Ext^2(\mathbf{I}_{[k+1,1]},\mathbf{I}_{[k,1]})$, and $[\epsilon_r]\neq [\epsilon_t]$ whenever $r\neq t$. Similar, one can show that for any nonzero $[\delta]\in \Ext^2(\mathbf{I}_{[k+1,1]},\mathbf{I}_{[k,1]})$, one has $[\delta]=[\epsilon_r]$ for some positive integer $r$. It follows that $\Ext^2(\mathbf{I}_{[k+1,1]},\mathbf{I}_{[k,1]})=\{0,[\epsilon_r]\mid r\geq 1\}\cong \mathbb{N}$.
\end{proof}
\subsection{The vanishing of \texorpdfstring{$\Ext^3(-,-)$}{Ext3}}

\begin{proposition}
    For any $\mathbf{M},\mathbf{N}\in \opname{rep}(\Delta_n,\mathbb{F}_1)_{\rm nil}$, $\Ext^3(\mathbf{M},\mathbf{N})=0$.
\end{proposition}
\begin{proof}
    By Lemma \ref{l:equiv=0} and Corollary \ref{c:decom-seq}, it suffices to assume that $\mathbf{M}$ is indecomposable, say $\mathbf{M}=\mathbf{I}_{[k,r]}$. Let 
    \[
    \xymatrix{\epsilon:&\mathbf{N}\,\ar@{>->}[r]^{f_0}&\mathbf{N}_1\ar[r]^{f_1}&\mathbf{N}_2\ar[r]^{f_2}&\mathbf{N}_3\ar@{->>}[r]^{f_3}&\mathbf{M}}
    \]
    be an arbitrary exact sequence in $\mathbb{E}^3(\mathbf{M},\mathbf{N})$. We may assume that $\epsilon$ is primitive. According to Lemma \ref{l:decom-epi} and Proposition \ref{prop:structure-nil-ind} (5), we conclude that $\mathbf{N}_3\cong \mathbf{I}_{[k,l]}$ for some $l\geq r$, and $\ker f_3\cong \mathbf{I}_{[k+r,l-r]}$. Applying Lemma \ref{l:decom-epi} and Proposition \ref{prop:structure-nil-ind} again, we obtain that $\mathbf{N}_2\cong \mathbf{I}_{[k+r,m]}$ for some $m\geq l-r$, and $f_2$ factors as $\mathbf{I}_{[k+r,m]}\twoheadrightarrow \mathbf{I}_{[k+r,l-r]} \rightarrowtail \mathbf{I}_{[k,l]}$. Consequently, We have the following commutative diagram of exact sequences:
    \[\xymatrix{\delta:&\mathbf{N}\ar@{=}[d]\ar@{=}[r]&\mathbf{N}\ar[d]^{f_0}\ar[r]^-0 &\mathbf{I}_{[k+r,m]}\ \ar@{=}[d]\ar@{>->}[r]&\mathbf{I}_{[k,m+r]}\ar@{->>}[d]\ar@{->>}[r]&\mathbf{I}_{[k,r]}\ar@{=}[d]\\
\epsilon:&
\mathbf{N}\ \ar@{>->}[r]^{f_0}&\mathbf{M}_1\ar[r]^-{f_1}&\mathbf{I}_{[k+r,m]}\ar[r]^{f_2}&\mathbf{I}_{[k,l]}\ar@{->>}[r]^{f_3}&\mathbf{I}_{[k,r]}.
}
\]
It is routine to see that $[\delta]=0$, and hence $[\epsilon]=0$.
\end{proof}

\bibliographystyle{plain}
\bibliography{ref}
\end{document}